\documentclass[10pt]{article}

\usepackage{amsmath}
\usepackage{amssymb}
\usepackage{amsthm}
\usepackage{graphics}
\usepackage[latin1]{inputenc}
\usepackage[T1]{fontenc}

\usepackage{enumerate}

\newtheorem{thm}{Theorem}[section]
\newtheorem{prop}[thm]{Proposition}
\newtheorem{coro}[thm]{Corollary}
\newtheorem{lemma}[thm]{Lemma}
\newtheorem{rem}[thm]{Remark}

\newcommand{\R}{\mathbb{R}}             
\newcommand{\N}{\mathbb{N}}             
\newcommand{\C}{\mathbb{C}}             

\newcommand{\e}{\epsilon}
\newcommand{\K}{\kappa}

\newcommand{\half}{\frac{1}{2}}
\renewcommand{\o}{\omega}

\newcommand{\ds}{\displaystyle}

\newcommand{\Section}[1]{\section{#1} \setcounter{equation}{0}}

\begin{document}

\title{Stability in the inverse Steklov problem on warped product Riemannian manifolds}
\author{Thierry Daud\'e \footnote{Research supported by the French National Research Projects AARG, No. ANR-12-BS01-012-01, and Iproblems, No. ANR-13-JS01-0006} $^{\,1}$, Niky Kamran \footnote{Research supported by NSERC grant RGPIN 105490-2011} $^{\,2}$ and François Nicoleau \footnote{Research supported by the French GDR Dynqua} $^{\,3}$\\[12pt]
 $^1$  \small D\'epartement de Math\'ematiques. UMR CNRS 8088, Universit\'e de Cergy-Pontoise, \\
 \small 95302 Cergy-Pontoise, France. \\
\small Email: thierry.daude@u-cergy.fr \\
$^2$ \small Department of Mathematics and Statistics, McGill University,\\ \small  Montreal, QC, H3A 2K6, Canada. \\
\small Email: nkamran@math.mcgill.ca \\
$^3$  \small  Laboratoire de Math\'ematiques Jean Leray, UMR CNRS 6629, \\ \small 2 Rue de la Houssini\`ere BP 92208, F-44322 Nantes Cedex 03. \\
\small Email: francois.nicoleau@math.univ-nantes.fr }



\date{\today}


\maketitle


\begin{abstract}

In this paper, we study the amount of information contained in the Steklov spectrum of some compact manifolds with connected boundary
equipped with a warped product metric. Examples of such manifolds can be thought of as deformed balls in $\R^d$. We first prove that the Steklov spectrum determines uniquely
the warping function of the metric. We show in fact that the approximate knowledge (in a given precise sense) of the Steklov spectrum is enough to determine uniquely the warping function in a neighbourhood of the boundary. Second, we provide stability estimates of log-type on the warping function from the Steklov spectrum.
The key element of these stability results relies on a formula that, roughly speaking, connects the inverse data (the Steklov spectrum) to the \emph{Laplace transform} of the difference of the two warping factors.


\vspace{0.5cm}




\end{abstract}

\tableofcontents


\Section{The model and statement of the results} \label{1}

In this paper, we shall consider a class of $d$-dimensional manifolds
\begin{equation} \label{Manifold}
  M = (0,1] \times S^{d-1},
\end{equation} 	
where $d \geq 3$  and with a connected boundary consisting  in a copy of $S^{d-1}$, \textit{i.e.} $\partial M = \{1\} \times S^{d-1}$. We shall assume these manifolds are equipped with
Riemannian warped product metrics of the form
\begin{equation} \label{Metric}
  g = c^4(r) [dr^2 + r^2 g_{S}].
\end{equation}
In what follows, $g_{S}$ denotes a {\it{fixed}} smooth Riemannian metric on the sphere $S^{d-1}$ and $c$ is a positive $C^{m}$-function, (with $m\geq 2$), of the variable $r$ only.
Without loss of generality, we may assume $c(0)=1$.

\vspace{0.2cm}


We emphasize
that under these general assumptions, the metrics (\ref{Metric}) are not necessarily regular. Actually, one proves that the metrics $g$ are regular if and only if the odd-order derivatives
$c^{(2k+1)}(0)=0$ and $g_{S}= d\Omega^2$ where $d\Omega^2$ is the round metric on $S^{n-1}$, (see \cite{Pe}, Section 4.3.4).
Nevertheless, in this paper, we consider general metrics of the form (\ref{Metric}) which could be singular at $r=0$, but emphasizing that in the regular case, the proofs of our results are much simpler.


\vspace{0.1cm}
It is convenient to use the change of coordinates $ x = -\log r \in [0, +\infty[$. In these new coordinates, the metric $g$ has the form
\begin{equation} \label{Metricx}
  g = f^4(x) [dx^2 +  g_{S}],
\end{equation}
where $f(x) = c(e^{-x}) e^{-\frac{x}{2}}$. Using the Taylor expansion of $c(r)$ at $r=0$, we see that the conformal factor $f(x)$ satisfies the following asymptotic expansion:
\begin{equation} \label{f}
  f(x) =  e^{-\frac{x}{2}} + \sum_{k=1}^m c_k e^{-(k+\half)x} + o( e^{ -(m+\frac{1}{2})x})  \ , \ x \to +\infty \  ,
\end{equation}
for some suitable real constants $c_k$.

\vspace{0.1cm}
A typical example of Riemannian manifold belonging to our class is the usual Euclidean metric on the unit ball of $\R^d$, obtained by simply taking $c(r)=1$, (i.e $f(x) = e^{-\frac{x}{2}}$),
and $g_S = d\Omega^2$. Hence, under our hypothesis, the Riemannian manifold $(M,g)$ can be viewed topologically as a unit ball of $\R^n$ whose metric is a deformation of the usual Euclidean metric both in the radial direction $r$ through the warping function $c(r)$, and also in the transversal directions through the general metric $g_S$.

\vspace{0.1cm}
In this paper, we are interested in studying the amount of information contained in the Steklov spectrum associated to the class of Riemannian manifolds $(M,g)$. Recall \cite{GP} that the Steklov
spectrum is defined as the spectrum of the Dirichlet-to-Neumann map $\Lambda_g$ (abbreviated later by DN map) associated to $(M,g)$. More precisely, consider the Dirichlet problem
\begin{equation} \label{Dirichlet}
  \left\{ \begin{array}{cc}
	-\triangle_g u = 0, & \textrm{on} \ M, \\
	u = \psi, & \textrm{on} \ \partial M,
	\end{array} \right.
\end{equation}
where $\psi \in H^{\half}(\partial M)$. Using the separation of variables, we shall prove in Section 2 that we can solve uniquely (\ref{Dirichlet}) even in the case where the metric $g$ is singular. Of course, when the metric is smooth, this result is well-known (see \cite{Sa, Ta1}) and the unique solution $u$ of (\ref{Dirichlet}) belongs to the Sobolev space $H^1(M)$.

\vspace{0.2cm}
In the latter case, we define the DN map $\Lambda_g$ as the operator $\Lambda_{g}$ from $H^{1/2}(\partial M)$ to $H^{-1/2}(\partial M)$ as
\begin{equation} \label{DN}
  \Lambda_{g} \psi = \left( \partial_\nu u \right)_{|\partial M},
\end{equation}
where $u$ is the unique solution of (\ref{Dirichlet}) and $\left( \partial_\nu u \right)_{|\partial M}$ is the normal derivative of $u$ with respect to the outer unit normal vector
$\nu$ on $\partial M$. Here $\left( \partial_\nu u \right)_{|\partial M}$ is interpreted in the weak sense as an element of $H^{-1/2}(\partial M)$ by
$$
  \left\langle \Lambda_{g} \psi | \phi \right \rangle = \int_M \langle du, dv \rangle_g \, dVol_g,
$$
for any $\psi \in H^{1/2}(\partial M)$ and $\phi \in H^{1/2}(\partial M)$ such that $u$ is the unique solution of (\ref{Dirichlet}) and $v$ is any element of $H^1(M)$ such
that $v_{|\partial M} = \phi$. If $\psi$ is sufficiently smooth, we can check that
$$
  \Lambda_{g} \psi = g(\nu, \nabla u)_{|\partial M} = du(\nu)_{|\partial M} = \nu(u)_{|\partial M},
$$
so that the expression in local coordinates for the normal derivative is thus given by
\begin{equation} \label{DN-Coord}
\partial_\nu u = \nu^i \partial_i u.
\end{equation}

The DN map is a pseudo-differential operator of order $1$ and is self-adjont on $L^2(\partial{M}, dS_g)$ where $dS_g$ denotes the metric induced by $g$ on the boundary $\partial M$.
Therefore, the DN map has a real and discrete spectrum accumulating at infinity. We shall thus denote the Steklov eigenvalues (counted with multiplicity) by
\begin{equation} \label{Steklov}
  0 = \sigma_0 < \sigma_1 \leq \sigma_2 \leq \dots \leq \sigma_k \to \infty.
\end{equation}
We refer the reader to the nice survey \cite{GP} and references therein. The main results of this paper are the following. First, we obtain the complete asymptotics of the Steklov spectrum $\sigma_k$ as $k \to \infty$ in terms of the Taylor series of the effective potential
\begin{equation} \label{qf}
  q_f(x) = \frac{(f^{d-2})''(x)}{f^{d-2}(x)} - \frac{(d-2)^2}{4},
\end{equation}
and of the eigenvalues of the positive Laplace-Beltrami operator of $(S^{d-1},g_S)$. Note that, in the radial coordinate $r$, the effective potential is given by:
\begin{equation}\label{potradial}
q_f(x) = (d-2) \left( (d-3) r^2 \left( \frac{c'(r)}{c(r)} \right)^2 + r^2 \frac{c''(r)}{c(r)} + (d-1) r \frac{c'(r)}{c(r)} \right).
\end{equation}
For later use, we denote the spectrum of $-\triangle_{g_S}$ (counting multiplicities) by
\begin{equation} \label{muk}
  0 = \mu_0 < \mu_1 \leq \mu_2 \leq \dots \leq \mu_k \to \infty.
\end{equation}
For each $k \geq 0$, the eigenvalue $\mu_k$ is associated to the normalized eigenfunction $Y_k \in L^2(S^{d-1}, dVol_{g_S})$ such that
$$
  -\triangle_{g_K} Y_k = \mu_k Y_k.
$$

\noindent
Let us finally introduce
\begin{equation} \label{kappak}
  \kappa_k = \sqrt{\mu_k + \frac{(d-2)^2}{4}}, \quad \forall k \geq 0,
\end{equation}
and recall the Weyl asymptotics (\cite{SV}, Theorem 1.2.1 and Remark 1.2.2)
\begin{equation} \label{WeylLaw}
  \kappa_k = c_{d-1}\, k^{\frac{1}{d-1}} + O(1), \quad k \to \infty.
\end{equation}
Here $c_{d-1} = 2\pi \ (\o_{d-1} Vol(S^{d-1}))^{-\frac{1}{d-1}}$ where $\o_{d-1}$ is the volume of the unit ball in $\R^{d-1}$, while $Vol(S^{d-1})$ is
the volume of $S^{d-1}$ for the metric $g_S$.


\vspace{0.1cm}\noindent
Our first result is the following:

\begin{thm} \label{AsympSteklov}
Let $(M,g)$ be a Riemannian manifold given by (\ref{Manifold})-(\ref{Metric}) and assume that $c \in C^{\infty}([0,1])$. Then the Steklov spectrum $(\sigma_k)_{k \geq 0}$ satisfies for all $N \in \N$,
\begin{equation} \label{AS1}
\sigma_k = \frac{(d-2) f'(0)}{f^3(0)} + \frac{\kappa_k}{f^2(0)} + \sum_{j=0}^N \frac{\beta_j(0)}{f^2(0)} \kappa_k^{-j-1} + O(\kappa_k^{-N-2}), \quad k \to \infty,
\end{equation}
where
\begin{equation} \label{AS2}
\left\{ \begin{array}{l}
\beta_0(x) = \half q_f(x), \\
{\displaystyle{\beta_{j+1}(x) = \half \beta_j'(x) + \half \sum_{l=0}^j \beta_l(x) \beta_{j-l}(x).}}
\end{array} \right.
\end{equation}
\end{thm}

\noindent We make the following observations: \\

%
%
%
%

\noindent
\textbf{1}. Using to (\ref{WeylLaw}), we find that the Steklov spectrum satisfies the usual Weyl law (see \cite{GP}):
\begin{equation}\label{WeylSteklov}
\sigma_k = \frac{c_{d-1}}{f^2(0)}\, k^{\frac{1}{d-1}} + O(1)= 2\pi \left( \frac{k}{\o_{d-1} Vol(\partial M)} \right)^{\frac{1}{d-1}}  + O(1), \quad k \to \infty.
\end{equation}
\textbf{2}. Note that the coefficients $\beta_j(0), \ j \geq 0$ only depend on the derivatives $q_f^{(l)}(0), \ l = 0, \dots, j$ up to order $j$. Hence the values $f(0), f'(0)$ and
the Taylor series at $0$ of the effective potential $q_f$ give the leading terms of the asymptotics of the Steklov spectrum in inverse powers of $\kappa_k$. But without additional knowledge of the asymptotics of the $\kappa_k$, we are not able to determine the warping function $f$ from this information.

\noindent \textbf{3}. In the exceptional case where $(S^{d-1},g_S) = (S^{d-1}, d\Omega^2)$, the spectrum of $-\triangle_{g_K}$ is given by $\{ \lambda_k = k(k+d-2), \ k \geq 0\}$ where each eigenvalue $\lambda_k$
has multiplicity $\frac{2k+d-2}{d-2} {k \choose k+d-3}$. Hence, if we order the Steklov spectrum $\sigma_k$ {\it{without counting mutiplicities}}, \textit{i.e.}
$$
  0 = \sigma_0 < \sigma_1 < \sigma_2 < \dots < \sigma_k \to \infty,
$$
where each $\sigma_k$ has multiplicity $\frac{2k+d-2}{d-2} {k \choose k+d-3}$ and denoting $\kappa_k = \sqrt{\lambda_k + \frac{(d-2)^2}{4}} = k+\frac{d-2}{2}$, we get the more precise asymptotics (without multiplicity):
\begin{equation} \label{AS3}
\sigma_k = \frac{(d-2) f'(0)}{f^3(0)} + \frac{k+\frac{d-2}{2}}{f^2(0)} + \sum_{j=0}^N \frac{\beta_j(0)}{f^2(0)} (k+\frac{d-2}{2})^{-j-1} + O(k^{-N-2}),
\quad k \to \infty .
\end{equation}
As a consequence, the asymptotics of the Steklov spectrum allow to recover inductively the values $f(0), f'(0)$ and the Taylor series at $0$ of the effective potential $q_f$.
Using then (\ref{qf}), we immediately see that in fact the Steklov spectrum determines the Taylor series of $f$. Hence we have proved:

\begin{coro} \label{UniquenessAnalytic}
  If $(S^{d-1},g_S) = (S^{d-1}, d\Omega^2)$ and if the warping function $c$ is analytic on $[0,1]$, then the Steklov spectrum of $(M,g)$ determines uniquely the warping function $c$.
\end{coro}

This leads to the question: does the Steklov spectrum determine uniquely the warping function $c$ without any assumption of analyticity on $c$ and on the transversal metric $g_S$ ? The answer is yes. Precisely, we have:

\begin{thm} \label{UniquenessSteklov}
  Let $(M,g)$ be a Riemannian manifolds given by (\ref{Manifold})-(\ref{Metric}). Then the Steklov spectrum $(\sigma_k)_{k \geq 0}$
  determines uniquely the warping function $c$.
\end{thm}


\vspace{0.1cm}
In dimension 2, the problem of recovering the metric from the Steklov spectrum on the unit disk was studied in a paper by Jollivet and Sharafutdinov \cite{JoSh}. They show that if the DN maps associated to the metrics $g_j$, $j=1,2$ are intertwined (which is a stronger hypothesis than the Steklov isospectrality), then the metrics $g_j$  are equal (up to a diffeomorphism and a conformal factor which is equal to $1$ on the boundary of the unit disk).

\vspace{0.2cm}
In fact, we prove an even better uniqueness result than Theorem \ref{UniquenessSteklov}: the subsequence $(\sigma_{k^{d-1}})$ determines uniquely the warping function $c$. The proof is based on the following \emph{local} uniqueness property:

\begin{thm} \label{LocalUniquenessSteklov}
 Let $(M,g)$ and $(M,\tilde{g})$ be Riemannian manifolds given by (\ref{Manifold})-(\ref{Metric}).
 Then, for a positive constant $a > 0$,  the two following assertions are equivalent:
\begin{eqnarray}
  \sigma_{k^{d-1}} - \tilde{\sigma}_{k^{d-1}} &=&  O(e^{-2a\kappa_{k^{d-1}}}), \quad k \to \infty. \label{Error} \\
 c(r)  &=&   \tilde{c}(r), \quad \forall r \in [e^{-a} ,1]. \label{LocalUniqueness}
\end{eqnarray}
\end{thm}

This theorem asserts that the knowledge of a subsequence of the Steklov spectrum up to an error that decays exponentially as $k \to \infty$ (see (\ref{Error})) determines uniquely the warping function $c$ in a neighbourhood of the boundary $r=1$ provided that $(S^{d-1},g_S )$ is known. Of course, Theorem \ref{UniquenessSteklov} is an immediate consequence of Theorem \ref{LocalUniquenessSteklov}. The question of whether the Steklov spectrum determines uniquely the warping function without the additional knowledge of the transversal Riemannian manifold $(S^{n-1},g_S)$ remains open.


\vspace{0.1cm}
Notice that the above result can be viewed as a weak form of a stability result. It stresses the important fact that the asymptotic behaviour of the Steklov spectrum allows one to determine the warping
function $c$ in a neighbourhood of the boundary $r=1$. In contrast, we shall show that the first Steklov eigenvalues determine in a stable
way the warping function $c$ on $[0,1]$. Before stating our main result, let us define our set $\mathcal{C}(A)$ of admissible warping functions $c(r)$, where $A$ is any positive constant, $m \geq 3$ and
$2 \leq p \leq m-1$:
\begin{equation}
\mathcal{C}(A) =\{ c \in C^m ([0,1])  \ {\rm{s.\ t.}} \ \ || c||_{C^m ([0,1])} + || \frac{1}{c} ||_{C^m ([0,1])}\leq A \ {\rm{and \ for}}\ k=1, ..., p-1, \ c^{(k)}(0) =0 \}.
\end{equation}


Roughly speaking, the conditions on the derivatives of $c(r)$ at $r=0$ tell us that the metric $g$ is relatively flat in a neighborhood of  the origin $r=0$.
These assumptions are only introduced to take into account the regularity of the warping function in the stability estimates obtained by integration by parts. We only assume $m \geq 3$ for the simplicity of the proofs. In particular, using (\ref{potradial}) we see that, if $ c \in \mathcal{C}(A)$,  the effective potential $q_f$ satisfies the uniform estimates:
\begin{equation}\label{estqA}
| q_f^{(k)} (x) | \leq C_A \  e^{-px} \ , \forall  x \geq 0, \ \forall k=0, ..., m-2,
\end{equation}
where the constant $C_A$ depends only on $A$.

\vspace{0.5cm}
In Section 4, we shall prove two log-type estimates results. The first one concerns the case where the metric $g$ is smooth, i.e we assume that the odd order derivatives of the warping functions
$c^{(2k+1)} (0) =0$, and we assume that the transversal metric  $g_S = d\Omega^2$ is the usual Euclidean metric on $S^{n-1}$. In this regular case, we have the following estimate:

\begin{thm} \label{LogStabSteklov}
Let $(M,g)$ and $(M,\tilde{g})$ be smooth  Riemannian manifolds given by (\ref{Manifold})-(\ref{Metric}) with  $c$, $\tilde{c} \in \mathcal{C}(A)$ where $A>0$ is fixed.
Assume  that for $\epsilon>0$ small enough, one has:
\begin{equation} \label{ErrorE}
  \sup_{k \geq 0} \  |\sigma_{k} - \tilde{\sigma}_{k}| \leq \e.
\end{equation}	
Then, there exists a positive constant $C_A$,  depending only on $A$ such that,
\begin{equation} \label{LogStability}
  || c-\tilde{c}||_{L^{\infty}(0,1)}  \leq C_A \ \left(\frac{1}{ \log ( \frac{1}{\e} )}\right)^{p-1}
\end{equation}
\end{thm}


\vspace{0.3cm}
This result is relatively close in spirit to the logarithmic stability estimates obtained by Alessandrini \cite{Ale}, and then improved by Novikov \cite{Nov}, for the Schr\"odinger equation on a bounded domain $M$ in $\R^d$. Indeed, it is well-known there is a close connection between the DN map which we have defined in the introduction for the metric $g = c^4(r)\  g_{e}$, ($g_{e}$ being the Euclidean metric on $\R^d$), and the DN map associated with the Schr\"odinger equation $(-\Delta_{g_e} +V) u=0$ where the potential $V$ is defined below. It is induced by the transformation law for the Laplace-Beltrami operator under conformal changes of metric:
\begin{equation} \label{ConformalScaling}
  -\Delta_{c^4 g_{e}} u = c^{-(d+2)} \left( -\Delta_{g_e} + V \right) \left( c^{d-2} u \right),
\end{equation}
where
\begin{equation} \label{q}
  V = c^{-d+2} \Delta_{g_e} c^{d-2}.
\end{equation}
\noindent
We recall that if $0$ is not an eigenvalue of $-\Delta_{g_e} + V$, then for all $ \psi \in H^{1/2}(\partial M)$, there exists a unique $u$ solution of
\begin{equation} \label{Eq0-Schrodinger}
  \left\{ \begin{array}{cc} (-\Delta_{g_e} + V) u = 0, & \textrm{on} \ M, \\ u = \psi, & \textrm{on} \ \partial M.
 \end{array} \right.
\end{equation}
The DN map for this Schr\"odinger equation is then given by $\Lambda_V (\psi) = \partial_{\nu}u_{|\partial M}$, where $\left( \partial_\nu u \right)_{|\partial M}$ is its normal derivative with respect to the unit outer normal vector $\nu$ on $\partial M$.

\vspace{0.2cm} \noindent
If the warping function $c(r)$ satisfies
$c_{|\partial M}=1$ and $\partial_{\nu}c_{|\partial M}=0$, (or equivalently  $c(1)=1$ and $c'(1)=0$ in our case), one can easily show (see for instance \cite{Sa}) that:
\begin{equation}\label{equivDN}
\Lambda_{g_e} = \Lambda_V.
\end{equation}
For the Schr\"odinger equation, Alessandrini \cite{Ale} has obtained the following stability estimate : assume that,  for $j=1,2$, the potentials $V_j$ belong to the Sobolev space $ W^{m,1}(M)$ with
$||V_j||_{W^{m,1}(M)} \leq A$. Then there exists a constant $C_A>0$ such that
\begin{equation}
|| V_1 -V_2 ||_{L^{\infty}(M)} \ \leq C_A \left( \log (3+ ||| \Lambda_{V_1} - \Lambda_{V_2}|||^{-1} \right)^{-\alpha},
\end{equation}
with $\alpha = 1 - \frac{d}{m}$ and where $||| B |||$ is the norm of the operator $B : L^{\infty}(\partial M) \rightarrow L^{\infty}(\partial M)$.
This stability estimate was then improved by Novikov in \cite{Nov}, where he proved one can take $\alpha = m-d$.

\vspace{0.2cm}
In Theorem \ref{LogStabSteklov}, under some additional assumptions on the derivatives of the warping function at $r=0$, we obtain a rather similar stability estimate with, roughly speaking,  $\alpha$ replaced by  $p-1$. Note that our potential V is only $C^{m-2}$ and that $p \leq m-1$. We emphasize that we do not assume that $c(1)=1$ and $c'(1)=0$, thus the connection between $\Lambda_g$ and $\Lambda_V$ is not so clear and depends implicitly on the unknown values $c(1)$ and $c'(1)$ which we are trying to estimate.


\vspace{0.5cm}
In the singular case, we also obtain a log-type estimate result, but which is much less precise. Indeed, in this case, we do not have  explicit formula for the angular eigenvalues $\kappa_k$, so we need to use
the Weyl's law to reasonably approximate the $\kappa_k$. We can prove the following theorem:

\begin{thm} \label{LogStabSteklovsing}
Let $(M,g)$ and $(M,\tilde{g})$ be singular Riemannian manifolds given by (\ref{Manifold})-(\ref{Metric}) with  $c$, $\tilde{c} \in \mathcal{C}(A)$ where $A>0$ is fixed.
Assume  that for $\epsilon>0$ small enough, one has:
\begin{equation} \label{ErrorE1}
  \sup_{k \geq 0} \ |\sigma_{k} - \tilde{\sigma}_{k}| \leq \e.
\end{equation}	
Then, there exists $\theta \in (0,1)$ and a positive constant $C_A$,  depending only on $A$ such that,
\begin{equation} \label{LogStability1}
  || c-\tilde{c}||_{L^{\infty}(0,1)}  \leq C_A \ \left(\frac{1}{ \log ( \frac{1}{\e} )}\right)^{(p-1)\theta}
\end{equation}
\end{thm}

\begin{rem}
 In the definition of our admissible warping functions set $\mathcal{C}(A)$, we impose that the integer $p\geq 2$. In particular, the first derivative at $r=0$
 of the warping function $c(r)$ must vanish. When the metric $g$ is smooth, this condition is automatically satisfied. When the metric $g$ is singular, we only make this assumption for simplicity of exposition in order to use the same strategy as for the regular case. This technical assumption could be certainly relaxed with a little effort, and one would obtain a logarithmic stabilily estimate close to the one stated in Theorem \ref{LogStabSteklovsing}.
\end{rem}

\vspace{0.2cm}

The rest of our paper is organized as follows. In Section \ref{2}, using the separation of variables, we recall the standard construction of the Dirichlet to Neumann map (see \cite{DKN2}, \cite{DKN3}, \cite{DKN4} for details), and some basics facts on the Steklov spectrum. In Section \ref{3}, we prove that the warping function $c(r)$ is uniquely determined by the knowledge of the Steklov spectrum. We  also obtain a local uniqueness result from the approximate knowledge of the Steklov spectrum. In Section \ref{4}, we prove our stability results: the proofs are based on the M\"untz-Jacskon approximation, starting from the knowledge of the Hausdorff moments.


\Section{The Steklov spectrum} \label{2}

In this section, we construct explicitly the Steklov spectrum by using the warped product structure of the manifold $(M,g)$. 
More precisely, we use the underlying symmetry of the warped product in order to diagonalize the DN map onto the Hilbert basis of 
harmonics $\{Y_k\}_{k \geq 0}$, \textit{i.e.} the normalized eigenfunctions, of $-\triangle_{g_S}$. On each harmonic, the DN map acts as an operator of multiplication by 
essentially the Weyl-Titchmarsh function associated to the countable family of Schr\"odinger operators arising from the separation of variables procedure. 
The Weyl-Titchmarsh theory will then allow us to prove the asymptotic of theorem \ref{AsympSteklov} as well the (local) uniqueness results of 
Theorems \ref{UniquenessSteklov} and \ref{LocalUniquenessSteklov} in Section \ref{3}.

Let us first solve the Dirichlet problem (\ref{Dirichlet}). In the coordinate system $(x,\o)$, the Laplace equation $-\triangle_g u = 0$ reads
\begin{equation} \label{LaplaceEq}
  [-\partial_x^2 - \triangle_{g_S} + q_f(x)] v = -\frac{(d-2)^2}{4} v,
\end{equation}
where $v = f^{d-2} u$ and $q_f$ is given by (\ref{qf}). Observe that under the hypothesis (\ref{f}), the effective potential $q_f$ is a smooth function on $[0,+\infty)$ that satisfies the asymptotics
\begin{equation} \label{AsympQf}
  q_f(x) = O(e^{- p x}), \quad x \to \infty.
\end{equation}

We now use the warped structure to separate variables. We thus look for solutions of (\ref{LaplaceEq}) of the form
\begin{equation} \label{SepaV}
  v = \sum_{k = 0}^\infty v_k(x) Y_k.
\end{equation}
Hence, for each $k \geq 0$, the functions $v_k$ satisfy the Schr\"odinger equation on the half-line
\begin{equation} \label{ODE-x}
  -v_k'' + q_f(x) v_k = -\kappa_k^2 v_k, \quad x \in [0,\infty),
\end{equation}
where $\kappa_k$ is given by (\ref{kappak}) for all $k \geq 0$. Actually, it is very useful to consider complex spectral parameter $z \in \C$ and we 
are interested by some special solutions of the Sturm-Liouville equation:
\begin{equation}\label{SLz}
-v'' + q_f (x) v = zv.
\end{equation}

We denote by $\{C_0(x,z), S_0(x,z)\}$ the fundamental system of solutions of (\ref{SLz}) with a spectral parameter $z \in \C$ that satisfy Neumann and Dirichlet conditions at $x=0$ respectively, given by
\begin{equation} \label{FSS}
  C_0(0,z) = 1, \ C_0'(0,z) = 0, \ S_0(0,z) = 0, \  S_0'(0,z) = 1.
\end{equation}
Note that
\begin{equation} \label{Wronskian}
  W(C_0(x,z),S_0(x,z)) = 1,
\end{equation}
where the Wronskian is defined by $W(u,v) = uv' - u'v$. Moreover, the functions $z \mapsto C_0(x,z), S_0(x,z)$ are entire in $z$. Note also that under the hypothesis (\ref{AsympQf}), the Schr\"odinger operator $H = -\frac{d^2}{dx^2} + q_f$ is in the limit point case at $x=\infty$. In consequence, for all 
$z \in \C$, there exists a unique (up to constant factor) solution $S_\infty(x,z)$ of (\ref{SLz}) that is $L^2$ in a neighbourhood of $x=\infty$, 
(see \cite{RS3}, Theorem XI.57 where our spectral parameter $z= k^2$). We write this function as
\begin{equation} \label{WeylSolution}
  S_\infty(x,z) = A(z) \left( C_0(x,z) - M(z) S_0(x,z) \right).
\end{equation}
Using (\ref{Wronskian}), we thus get the following expressions for the function $A(z)$ and the Weyl-Titchmarsh function $M(z)$
\begin{equation} \label{WT}
  A(z) = W(S_{\infty}(x,z), S_0(x,z)), \quad \quad M(z) = - \frac{W(C_0(x,z), S_\infty(x,z))}{W(S_0(x,z), S_\infty(x,z))} = \frac{S'_\infty(0,z)}{S_\infty(0,z)}.
\end{equation}
For later use, we also introduce the characteristic functions
\begin{equation} \label{Char}
  \Delta(z) = W(S_0(x,z), S_\infty(x,z)) = -A(z), \quad d(z) = W(C_0(x,z), S_\infty(x,z)).
\end{equation}
We thus have:
\begin{equation}\label{WTquotient}
M(z) = - \frac{d(z)}{\Delta(z)}.
\end{equation}
Notice that the characteristic functions $\Delta(z)$ and $d(z)$  are analytic on $\C$, and the WT function $M$ is analytic on $\C\backslash [\beta, +\infty[$ 
with $-\beta$ sufficiently large. The zeros $(-\alpha_j^2)_{j \geq 0}$ of the function $z \mapsto \Delta(z)$ are precisely the Dirichlet eigenvalues of the self-adjoint operator $H$ (and thus are real) whereas the zeros $(-\gamma_j^2)_{j \geq 0}$ of the function $z \mapsto d(z)$ are the Neumann eigenvalues of the self-adjont operator $H$ (and thus are real too). Moreover, we know that $\sigma_{ess}(H) = [0,+\infty)$ and that the essential spectrum contains no embedded eigenvalues (\cite{RS4}, Thm XIII.56). In consequence of the spectral theorem, the eigenvalues $(-\alpha_j^2)_{j \geq 0}$ and $(-\gamma_j^2)_{j \geq 0}$ must then satisfy
\begin{equation} \label{CondEigen}
 -\alpha_j^2, -\gamma_j^2 \in [\min(q_f), 0] \quad , \quad \forall j \geq 0,
\end{equation}
with the usual convention that $[\min(q_f), 0] = \emptyset$ if $\min(q_f)>0$.

We have a more precise result on the location of the eigenvalues $-\alpha_j^2$ of $H$ thanks to

\begin{lemma} \label{EigenH}
  The discrete Dirichlet spectrum of $H = -\frac{d^2}{d x^2} + q_f$ is finite and contained in $(-\frac{(d-2)^2}{4}, 0)$.
\end{lemma}

\begin{proof}
Introduce the operators $L = H + \frac{(d-2)^2}{4}$ and $K = f^{-d+2} L f^{d-2}$. An easy calculation shows that
$$
  K = -\frac{1}{f^{d-2}} \frac{d}{dx} \left( f^{d-2} \frac{d}{dx} \right).
$$
We remark that $K$ is selfadjoint on $L^2(\R^+; f^{2d-4}(x) dx)$ if we impose Dirichlet boundary condition at $x=0$. Moreover, the operator $K$ is clearly positive. Hence we have $\sigma_{pp}(K) \subset (0,+\infty)$. But $\lambda$ is an eigenvalue of $K$ if and only if $\lambda - \frac{(d-2)^2}{4}$ is an eigenvalue of $H$. In particular, we obtain that an eigenvalue $-\alpha_j^2$ of $H$ always satisfy $-\alpha_j^2 > -\frac{(d-2)^2}{4}$. Together with (\ref{CondEigen}), this proves the result.
\end{proof}

\begin{coro} \label{0NotEigenvalue}
  $0$ does not belong to the Dirichlet spectrum of $-\Delta_g$.
\end{coro}

\begin{proof}
  Assume the converse. Then there exists $u \ne 0$ such that $-\Delta_g u = 0$ on $M$ and $u=0$ on $\partial M$. Using separation of variables, this means that there exists $k \geq 0$, $v_k \ne 0$ and $v_k(0) = 0$ such that $H v_k = -\kappa_k^2 v_k$. In other words, $-\kappa_k^2$ is an eigenvalue of $H$ with Dirichlet boundary condition. According to Lemma \ref{EigenH}, we then must have $-\kappa_k^2 > -\frac{(d-2)^2}{4}$. But since, $\mu_k \geq 0$, we always have $-\kappa_k^2 \leq -\frac{(d-2)^2}{4}$. Contradiction.
\end{proof}

Finally, the DN map can be clearly diagonalized onto the Hilbert basis of harmonics $\{Y_k\}_{k \geq 0}$. If we represent the Dirichlet data as $\psi = \sum_{k \geq 0} \psi_k Y_k$, then the global DN map has the expression
\begin{equation} \label{GlobalDN}
  \Lambda_g \psi = \sum_{k =0}^\infty  (\Lambda_g^k \psi_k) Y_k,
\end{equation}
where the diagonalized DN map are defined by
\begin{equation} \label{DiagDN}
  \Lambda^k_g \psi_k = \frac{(d-2) f'(0)}{f^{d+1}(0)} v_k(0) - \frac{v'_k(0)}{f^d(0)}.
\end{equation}
Using that $v_k(x) = \alpha_k C_0(x,\kappa_k) + \beta_k S_0(x,\kappa_k)$, a straightforward calculation shows that the partial DN map $\Lambda_g^k$ acts as an operator of multiplication, precisely
\begin{equation} \label{DNk}
  \Lambda^k_g \psi_k = \left( \frac{(d-2) f'(0)}{f^3(0)} - \frac{M(-\kappa_k^2)}{f^2(0)} \right) \psi_k.
\end{equation}

We see immediately from (\ref{DNk}) that the partial DN map $\Lambda_g^k$ acts essentially by an operator of multiplication by the WT function $M(-\kappa_k^2)$ associated to (\ref{ODE-x})
up to some boundary values of the warping function $f$ ans its first derivative $f'$. In consequence, we infer that the Steklov spectrum of $(M,g)$, that is the set of
eigenvalues of $\Lambda_g$, is precisely given by
\begin{equation} \label{LinkSteklovWT1}
  \{\sigma_k, \ k \geq 0 \} = \left\{ \frac{(d-2) f'(0)}{f^3(0)} - \frac{M(-\kappa_j^2)}{f^2(0)}, \ j \geq 0 \right\}.
\end{equation}
In fact, we have the following exact identification:

\begin{lemma} \label{SteklovSpectrumWT}
  1. The function $y \in \R \to M(y)$ is strictly increasing on $\R \setminus \{-\alpha^2_j, \ j \geq 0\}$. \\
  2. Then $ \forall k \geq 0$, $\sigma_k = \frac{(d-2) f'(0)}{f^3(0)} - \frac{M(-\kappa_k^2)}{f^2(0)}$. \\
\end{lemma}

\begin{proof}
1. Let $y \in \R \backslash\{-\alpha_j²\}$. Notice then that $C_0(x,y), S_0(x,y), S_\infty(x,y)$ are real and thus $\Delta(y)$ and $d(y)$ are real too.
Let $y^* \in \R$. Since $(S_\infty(x,y) S'_\infty(x,y^*) - S'_\infty(x,y) S_\infty(x,y^*))' = (y-y^*)S_\infty(x,y) S_\infty(x,y^*)$, and using that for $k=0,1$,
$\ds\lim_{x \to \infty} S_\infty^{(k)}(x,y) = 0$
and (\ref{Char}), we have:
$$
  (y-y^*) \int_0^\infty S_\infty(x,y) S_\infty(x,y^*) dx = \Delta(y) d(y^*) - \Delta(y^*) d(y).
$$
Hence letting $y^* \to y$, we get
$$
  \int_0^\infty S_\infty(x,y)^2 dx = d(y) \dot{\Delta}(y) - \Delta(y) \dot{d}(y) = \Delta^2(y) \dot{M}(y),
$$
where $\dot{}$ denotes the derivative with respect to $y$. We conclude from this that 1. holds. Observe that between two Dirichlet eigenvalues $-\alpha_j^2 < -\alpha_{j+1}^2$,
the WT function $M(y)$ is strictly increasing and goes from from $-\infty$ to $+\infty$.

2. From Lemma \ref{EigenH}, we know that the Dirichlet eigenvalues of $H$ satisfy $-\frac{(d-2)^2}{4} < -\alpha_j^2 \leq 0$ for all $j \geq 0$. In other words, we have
$$
  \forall j \geq 0, \quad 0 \leq \alpha_j < \frac{d-2}{2}.
$$
Since $\kappa_k \geq \frac{d-2}{2}$ for all $k \geq 0$ (see (\ref{kappak})), we deduce from 1. that the function $\kappa_k \mapsto M(-\kappa_k^2)$ is strictly decreasing
for $k \geq 0$. Hence 2. follows from (\ref{LinkSteklovWT1}) and the ordering of the Steklov spectrum $(\sigma_k)_{k \geq 0}$.

\end{proof}

Having obtained the exact expression of the Steklov spectrum in terms of the WT functions (\ref{WT}) evaluated at the $-\kappa_k^2$, we can prove easily Theorem \ref{AsympSteklov}
by using the well-known asymptotic of a WT function due to Danielyan and Levitan \cite{DL}. We refer for instance to \cite{Si1}, Section 4 and the references therein for a proof and more asymptotic results.

\begin{proof} [Proof of Theorem \ref{AsympSteklov}]
Recall the asymptotic of the WT function from \cite{Si1}, Thm 4.5. If $q_f \in C^N([0,\delta))$ with $\delta > 0$, then as $\kappa \to \infty$, we have
\begin{equation} \label{AsympWT}
  M(-\kappa^2) = - \kappa - \sum_{j=0}^N \beta_j(0) \kappa^{-j-1} + O(\kappa^{-N-2}),
\end{equation}	
where the constants $\beta_j(0)$ can be calculated inductively by (\ref{AS2}). Since the potential $q_f$ is assumed to be smooth, the asymptotic (\ref{AsympWT}) together with lemma \ref{SteklovSpectrumWT} lead to the result.
\end{proof}


\Section{Uniqueness and local uniqueness} \label{3}

In this section, we prove local uniqueness, that is Theorem \ref{LocalUniquenessSteklov}, under the assumption that the transversal Riemanniann manifold $(S^{n-1},g_S)$ is known. The general idea of local uniqueness inverse results arises in the different versions of the local Borg-Marchenko Theorem stated first by Simon in \cite{Si1} and proved differently or extended to singular settings in \cite{Be, ET, GS1, GS2, KST}. We shall follow here the initial version due to Simon that makes an intensive use of a representation of the WT function $M$ as the Laplace transform of what Simon called the $A$-function. Precisely, Simon showed in \cite{Si1}, Thm 2.1, that there exists a function $A$ on $[0,\infty)$ with $A-q_f$ continuous, obeying
\begin{equation} \label{EstA1}
  |A(\alpha) - q_f(\alpha)| \leq Q(\alpha)^2 e^{\alpha Q(\alpha)}, \quad Q(\alpha) = \int_0^\alpha |q_f(s)|ds,
\end{equation}
such that, if $\kappa > \half \|q_f\|_{L^1}$, then
\begin{equation} \label{ARepresentation}
  M(-\kappa^2) = -\kappa - \int_0^\infty A(\alpha) e^{-2\kappa\alpha} d\alpha.
\end{equation}
We also have
\begin{equation} \label{EstA2}
  |A(\alpha,q_f) - A(\alpha,\tilde{q_f})| \leq \|q_f - \tilde{q_f}\|_{L^1} [Q(\alpha) + \tilde{Q}(\alpha)] e^{\alpha [Q(\alpha)+\tilde{Q}(\alpha)]}.
\end{equation}
Finally, Simon also proved the local uniqueness result

\begin{thm}[\cite{Si1}, Thm 1.5] \label{UniquenessA}
  The potential $q_f$ on $[0,a]$ is a function of $A$ on $[0,a]$. Explicitly, if $q_f$ and $\tilde{q_f}$ are two potentials, let $A$ and $\tilde{A}$ be their $A$-functions. Then
$$
  A(\alpha) = \tilde{A}(\alpha), \ \forall \alpha \in [0,a] \ \Longleftrightarrow \ q_f(x) = \tilde{q_f}(x), \ \forall x \in [0,a].
$$
\end{thm}

We shall use these results as follows.

\begin{proof}[Proof of theorem \ref{LocalUniquenessSteklov}]
Suppose that the assumption (\ref{Error}) is satisfied. Using Theorem \ref{AsympSteklov} and Lemma \ref{SteklovSpectrumWT}, we see immediately that
\begin{equation} \label{IniCond}
  f(0) = \tilde{f}(0), \quad f'(0) = \tilde{f}'(0).
\end{equation}
Hence the assumption (\ref{Error}) can be equivalently read as
\begin{equation} \label{Error1}
  M(-\kappa_{k^{d-1}}^2) - \tilde{M}(-\kappa_{k^{d-1}}^2) = O(e^{-2a\kappa_{k^{d-1}}}), \quad k \to \infty.
\end{equation}
We use then the representation of the WT functions (\ref{ARepresentation}), the behavior (\ref{AsympQf}) of the potentials $q_f,\, q_{\tilde{f}}$ and the estimate (\ref{EstA2}) to show that (\ref{Error1}) entails
\begin{equation} \label{LapTrans1}
  \int_0^a \left[ A(\alpha) - \tilde{A}(\alpha)\right] e^{-2\kappa_{k^{d-1}} \alpha} d\alpha = O(e^{-2a\kappa_{k^{d-1}}}), \quad k \to \infty.
\end{equation}

\noindent
Now, we need the following proposition which is a slight generalization to the case of noninteger $(\kappa_k)_{k \geq 0}$ of Proposition 2.4 in \cite{Hor1}.


\begin{prop}\label{uniciteLaplace}
Let $f \in L^1(0,a)$. Assume that
\begin{equation}\label{hypothesis}
\int_0^a \ e^{-\kappa_{k^{d-1}} t}  f(t)  \ dt \ =\ O(e^{-a\kappa_{k^{d-1}}})\ , \ \ k \rightarrow +\infty.
\end{equation}
Then, $f=0$ almost everywhere.
\end{prop}

\begin{proof}
Setting  $\lambda_k = \frac{1}{c_{d-1}} \kappa_{k^{d-1}}$, we deduce from (\ref{WeylLaw}) there exists $C>0$ such that $|\lambda_k - k | \leq  C$.
\vspace{0.2cm}
\par\noindent
For a fixed $N \in \N$ large enough, we set $\nu_k = \frac{\lambda_{kN}}{N}$, and we have $|\nu_k -k | \leq \frac{C}{N}<\frac{1}{4}$. Using (\ref{hypothesis}), a straightforward calculation gives:
\begin{equation}\label{Laplace}
\int_0^b g(y) e^{-\nu_k y} \ dy = O(e^{-b \nu_k}) , \ k \rightarrow + \infty,
\end{equation}
where we have set ${\displaystyle{b = c_{d-1} Na}}$, and ${\displaystyle{g(y) = f(\frac{y}{c_{d-1} N})}}$. Now, let us define for $z \in \C$,
\begin{equation}
F(z) = e^{bz} \int_0^b g(y) e^{-z y} \ dy.
\end{equation}
Clearly, $F(z)$ is an entire function which obeys
\begin{equation}
|F(z)| \leq ||g||_1 \  e^{b Re_+(z)},
\end{equation}
where $Re_+ (z)$ is the positive part of $Re \ z$. Moreover, from (\ref{Laplace}), we get $ F(\nu_k) = O(1)$. It follows from a theorem of Duffin and Schaeffer
(\cite{Boa}, Theorem 10.5.1) that $F(x)$ is bounded for $x>0$, or equivalently
\begin{equation}
\int_0^b g(y) e^{-x y} \ dy = O(e^{-bx}) \ ,\ x \rightarrow +\infty.
\end{equation}
Then, using (\cite{Si1}, Lemma A.2.1), we have  $g = 0$ almost everywhere in $(0,b)$ which concludes the proof of this Proposition.

\end{proof}


\noindent
Thus, Proposition \ref{uniciteLaplace} implies in particular that
$$
  A(\alpha) = \tilde{A}(\alpha), \quad \forall \alpha \in [0,a],
$$
from which we infer using Theorem \ref{UniquenessA} that
$$
  q_f(x) = q_{\tilde{f}}(x), \quad \forall x \in [0,a].
$$
Recalling the definition (\ref{qf}) of $q_f$, we thus see that
\begin{equation} \label{a1}
  (f^{d-2})''(x) = \frac{(\tilde{f}^{d-2})''(x)}{\tilde{f}^{d-2}(x)} f^{d-2}(x), \quad \forall x \in [0,a]
\end{equation}
We finish the proof seeing that (\ref{a1}) can be viewed as a linear second-order ODE for $f^{d-2}$. Recalling from (\ref{IniCond}) that the Cauchy data (\ref{IniCond}) of $f^{d-2}$
are equal to those of $\tilde{f}^{d-2}$, we conclude that the unique solution of (\ref{a1}) on $[0,a]$ is $f^{d-2} = \tilde{f}^{d-2}$. Whence the asserted result.

\vspace{0.5cm}
Conversely, assume that (\ref{LocalUniqueness}) holds. In particular,  $q_f(x)=q_{\tilde{f}}(x)$ for all $x \in (0,a)$. So, using Theorem \ref{UniquenessA}, we get $A(\alpha) = \tilde{A}(\alpha)$ for all $\alpha \in (0,a)$. Then, using the same arguments as in the first part of the proof, we obtain
\begin{equation}
M(-\kappa^2) = \tilde{M}(-\kappa^2) + O(e^{-2a\kappa}) \ \ , \ \ \kappa \rightarrow + \infty.
\end{equation}
Then, using Lemma \ref{SteklovSpectrumWT} and noting that $f(0), f'(0)$ are also known, we get immediately (\ref{Error}).

\end{proof}


\Section{Proof of Theorems \ref{LogStabSteklov} and \ref{LogStabSteklovsing}} \label{4}

\subsection{A Volterra  type integral operator}

Let us begin by an elementary lemma:

\begin{lemma}\label{eqint}
For any $\K$ large enough, we have:
\begin{equation*}
S_{\infty}(0,-\K^2) \tilde{S}_{\infty}(0,-\K^2) \left( M(-\K^2) - \tilde{M}(-\K^2) \right)=
\int_0^{+\infty} (q_{\tilde{f}}(x) -q_f(x)) \ S_{\infty}(x,-\K^2) \tilde{S}_{\infty}(x,-\K^2) \ dx.
\end{equation*}
\end{lemma}

\begin{proof}
To simplify the notation, we set $z = -\K^2$ and we integrate over the interval $(0, \infty)$ the  obvious equality:
\begin{equation}
\left( S_{\infty}(x,z) \tilde{S}_{\infty}'(x,z)- S_{\infty}'(x,z) \tilde{S}_{\infty}(x,z) \right)' =
(q_f (x)- q_{\tilde{f}} (x)) S_{\infty}(x,z) \tilde{S}_{\infty}(x,z).
\end{equation}
For $k=0$ or $1$,  $S_{\infty}^{(k)} (x,z) \rightarrow 0$ as $x \rightarrow +\infty$, (see for instance \cite{BS}), so we get immediately
\begin{equation}
 S_{\infty}'(0,z) \tilde{S}_{\infty}(0,z) - S_{\infty}(0,z) \tilde{S}_{\infty}'(0,z) =
\int_0^{+\infty} ( q_f (x) - q_{\tilde{f}} (x)) S_{\infty}(x,z) \tilde{S}_{\infty}(x,z) \ dx,
\end{equation}
which implies the Lemma, thanks to (\ref{WT}).
\end{proof}

\vspace{0.3cm}
In the following lemma, we express the Weyl solution $S_{\infty}(x,z)$ with the help of the  well-known Marchenko's representation,
(we refer to \cite{Ma}, Chapter III for details):

\begin{lemma}\label{Marchenko}
Assume that $c \in \mathcal{C}(A)$. Then, there exists a $C^{m-1}$ function $K(x,t)$ for $0 \leq x \leq t < \infty$,  satisfying the properties:
\begin{eqnarray}\label{eqMarchenko}
S_{\infty}(x, -\K^2) &=& e^{-\K x} + \int_x^{+\infty} K(x,t) e^{-\K t} \ dt \ ,\ \K > 0. \\
K(x,x) &=& \frac{1}{2} \int_x^{+\infty} q_f (t) \ dt.
\end{eqnarray}
Moreover, there exists a constant $C_A >0$ depending only on $A$ such that,
\begin{equation}\label{estkernelK}
| \partial_x^k \partial_t^l K(x,t)| \leq C_A \ e^{- \frac{p}{2}(x+t)} \ ,\ \forall  k,l \leq m-1 .
\end{equation}
\end{lemma}

\begin{proof}
The existence of the Marchenko's kernel $K(x,t)$ is proved in (\cite{Ma}, Lemma 3.1.1), and we have the following estimate:
\begin{equation}
| K(x,t)| \leq \frac{1}{2} \  \sigma( \frac{x+t}{2}) \ e^{\sigma_1 (x) - \sigma_1 (\frac{x+t}{2})},
\end{equation}
where
\begin{equation}
\sigma(x) = \int_x^{+\infty} |q_f(s)| \ ds \ , \ \sigma_1 (x) = \int_x^{+\infty} \sigma(s) \ ds.
\end{equation}
Thus, using (\ref{estqA}), we see that $|K(x,t)|\leq C_A e^{- \frac{p}{2}(x+t)}$, i.e we have proved (\ref{estkernelK}) in the case $k=l=0$. Now, let us define
$H(u,v) = K(u-v, u+v)$ for $0 \leq v \leq u$. Thanks to (\cite{Ma}, Lemma 3.1.2), $H$ obeys:
\begin{eqnarray}
\frac{\partial H}{\partial u} (u,v) & =& - \frac{1}{2} q_f (u) - \int_0^v q_f (u-s) H(u,s) \ ds, \\
\frac{\partial H}{\partial v} (u,v) & =& \int_u^{+\infty}  q_f (s-v) H(s,v) \ ds.
\end{eqnarray}
Then, (\ref{estkernelK}) follows from a straightforward calculation.
\end{proof}

\vspace{0.2cm}
\noindent
As a by-product, we get:

\begin{coro}\label{estSinfini}
Let $c \in \mathcal{C}(A)$ be a warping function for the metric (\ref{Metric}). Then, there exists a constant $C_A$ such that,
\begin{equation*}
|S_{\infty}(0, -\K^2) -1| \leq \frac{C_A}{\K +1} \ ,\ {\rm{for \ all}} \ \K\geq 0.
\end{equation*}
\end{coro}

\begin{proof}
Using (\ref{estkernelK}) for $k=p=0$, we get:
\begin{equation*}
 S_{\infty}(0, -\K^2) -1 = \int_0^{+\infty} K(0,t) \ e^{-\K t} \ dt  \leq \frac{C_A}{\K +1}.
\end{equation*}
\end{proof}

\vspace{0.5cm}\noindent
Now, let us  introduce  a new  kernel  $K_1(x,t)$ for $0 \leq t \leq x < \infty$, by the formula:
\begin{equation}
K_1 (x,t) = 2 K(t, 2x-t) + 2 \tilde{K} (t, 2x-t) + 2 \int_t^{2x-t} K(t,u) \tilde{K} (t, 2x-u) \ du,
\end{equation}
where $\tilde{K}(x,t)$ is the Marchenko's kernel associated with the potential $q_{\tilde{f}}$, (see Lemma \ref{Marchenko}). We have the following estimate which follows immediately from
(\ref{estkernelK}):

\begin{lemma}\label{derivK1}
Assume that $c \in \mathcal{C}(A)$. Then, for all $\alpha <p$, there exists a constant $C_{A,\alpha} >0$ depending only on $A$ and $\alpha$ such that,
\begin{equation}\label{estkernelK1}
| \partial_x^k \partial_t^l K_1 (x,t)| \leq C_{A,\alpha} \ e^{- \alpha x} \ ,\ \forall  k,l \leq m-1 .
\end{equation}

\end{lemma}

\vspace{0.5cm}\noindent
Finally, we consider the corresponding Volterra type integral operator $B$ given by:
\begin{equation}
Bh(x) = h(x) + \int_0^x K_1 (x,t) h(t) \ dt.
\end{equation}

\noindent
This operator $B$ is crucial to our analysis because it links the Steklov spectrum to the difference of the potentials $q_f$ and $q_{\tilde{f}}$.
More precisely, one has the following result:

\begin{lemma}
For $\K$ sufficiently large, we have:
\begin{equation}\label{intM}
S_{\infty}(0, -\K^2) \tilde{S}_{\infty}(0, -\K^2) \left(M(-\K^2) - \tilde{M}(-\K^2) \right) = \int_0^{+\infty}  e^{-2\K x} B[q_{\tilde{f}} -q_f] (x) \ dx.
\end{equation}
\end{lemma}

\begin{proof}
Thanks to (\cite{Hor2}, Lemma 2.5), we have:
\begin{equation}\label{eqHorvath}
\int_0^{+\infty} ( q_f (x) - q_{\tilde{f}} (x)) S_{\infty}(x,-\K^2) \tilde{S}_{\infty}(x,-\K^2) \ dx =
\int_0^{+\infty} e^{-2\K x} B[q_{\tilde{f}} -q_f] (x) \ dx.
\end{equation}
Then, using Lemma \ref{eqint} and (\ref{eqHorvath}), we get (\ref{intM}).
\end{proof}

This Volterra operator $B$ also possesses good properties on some $L^2$-spaces equipped with exponential weights, and which are defined by:
\begin{equation}\label{Cdelta}
\mathcal{H}_{\delta} =\{q : ||q||_{\mathcal{H}_\delta}^2 := \int_0^{+\infty} |q(x) |^2 \ e^{\delta x} \ dx < \infty \}.
\end{equation}
If $c \in \mathcal{C}(A)$, then it results from (\ref{estqA}) that $q_f \in \mathcal{H}_{\delta}$ for any $\delta< 2p$. Moreover,
there exists a constant $C_{A, \delta}$ depending only on $A$ and $\delta$ such that
\begin{equation}\label{estqfA}
 ||q_f||_{\mathcal{H}_\delta} \leq C_{A, \delta} \ {\rm{for \ all}} \ c \in \mathcal{C}(A).
\end{equation}

\vspace{0.2cm}\noindent
In the  following Proposition, which is close to \cite{Hor2}, Lemmas 2.5 - 2.6, we give a uniform estimate on the norm of   $B: \mathcal{H}_{\delta} \rightarrow \mathcal{H}_{\delta}$, and its inverse,
when the warping functions belong to the admissible set $\mathcal{C}(A)$. This result will be very useful to estimate  the difference of the potentials
$q_{\tilde{f}} -q_f$, (for the topology of  $\mathcal{H}_{\delta}$).

\begin{prop}\label{isomorphism}
Let $c, \tilde{c}$ be warping functions belonging to $ \mathcal{C}(A)$. Then, for any $0<\delta < p$, we have:
\begin{equation*}
B :  \mathcal{H}_{\delta}   \rightarrow \mathcal{H}_{\delta} \ {\rm {is \ an \ isomorphism}},
\end{equation*}
and there exists a constant $C_{A, \delta}$ depending only on $A$ and $\delta$ such that
\begin{equation} \label{estB}
  || B || +  || B^{-1} || \leq C_{A, \delta}.
\end{equation}
\end{prop}

\begin{proof}
By convention, in what follows, $C_A$ or $C_{A, \delta}$ denote constants depending only on $A$, (or only on $A$ and $\delta$), which can differ from one line to the other.
\vspace{0.1cm}
\noindent
Let $\alpha \in ]\delta, p[$ be fixed. We split the operator $B$ as $B= Id + C$ where $ C $ is the Volterra operator given by
\begin{equation}
 Ch(x) = \int_0^x K_1 (x,t) \ h(t) \ dt
\end{equation}
For $h \in \mathcal{H}_{\delta}$, using (\ref{estkernelK1}), we get immediately
\begin{equation}
 |Ch(x)| \leq C_{A, \delta} \ e^{-\alpha x} ||h||_{\mathcal{H}_{\delta}},
\end{equation}
which clearly implies $||Ch||_{\mathcal{H}_{\delta}} \leq  C_{A, \delta}  \ ||h||_{\mathcal{H}_{\delta}}$. Then, $B :  \mathcal{H}_{\delta}   \rightarrow \mathcal{H}_{\delta}$ is bounded and we have
$||B|| \leq C_{A, \delta}$.

Now, let us  prove that $C : \mathcal{H}_{\delta} \rightarrow \mathcal{H}_{\delta}$ is a Hilbert-Schmidt operator. To this end, we calculate:
\begin{eqnarray*}
||C||_{HS}^2  &:= & \int_0^{+\infty} \int_0^{+\infty}  |K_1(x,t) {\bf{1}_{\{ t \leq x \}}} |^2 \ e^{\delta (x+t)} \ dx\ dt \\
            &=  & \int_0^{+\infty} \int_0^x  |K_1(x,t)  |^2 \ e^{\delta (x+t)} \ dt\ dx \\
            &\leq &  C_{A, \delta} \int_0^{+\infty} \int_0^x e^{-2\alpha x} \ e^{\delta (x+t)} \ dt\ dx \\
            &\leq &   C_{A, \delta} \int_0^{+\infty} x e^{-2\alpha x} \ e^{2\delta x} \ dt\ dx \\
            &\leq &  C_{A, \delta},
\end{eqnarray*}
since $\alpha \in ]\delta, p[$. It follows that $C$ is a Hilbert-Schmidt operator, and a fortiori $C$ is a compact operator.
So, if $B$ is not an isomomorphism, then $-1$ must be an eigenvalue of $C$. But this is impossible for a Volterra operator with continuous kernel.

\vspace{0.2cm}\noindent
Now, let us estimate the norm of the inverse operator $B : {\mathcal{H}_{\delta}} \rightarrow {\mathcal{H}_{\delta}} $. We denote by $K_n(x,t)$ the integral kernel of  the operator $C^n$, $n \geq 1$.
Clearly, these kernels  satisfy the relation
\begin{equation}
 K_{n+1}(x,t) = \int_t^x K_1(x,s) K_n (s,t) \ ds \ ,\ t \leq x,
\end{equation}
and we have the following estimates which can be easily proved by induction: for $\delta <2$,
\begin{equation}
 |K_n (x,t) | \leq  \left( C_{A, \delta} e^{-\delta x} \right)^n \ \frac{(x-t)^{n-1}}{(n-1)!} \ ,\ t\leq x.
\end{equation}
For $n \geq 2$, using the rough bound $||C^n|| \leq ||C^n||_{HS}$, we get:
\begin{eqnarray*}
 || C^n || &\leq&   \int_0^{+\infty} \int_0^x  \left( C_{A, \delta} e^{-\delta x} \right)^{2n} \ \left( \frac{(x-t)^{n-1}}{(n-1)!} \right)^2 \ e^{\delta(t+x)} \ dt   \ dx \\
           &\leq& \left( \frac{(C_{A, \delta})^{n}}{(n-1)!}\right)^2  \ \int_0^{+\infty} x^{2n-1} \ e^{-2(n-1)\delta x} \ dx \\
           &\leq & \left( \frac{(C_{A, \delta})^{n}}{(n-1)!}\right)^2 \left( \frac{1}{2(n-1)\delta} \right)^{2n} \ \int_0^{+\infty} y^{2n-1} e^{-y} \ dy \\
           &\leq& \left( \frac{C_{A, \delta}}{2(n-1)\delta}\right)^{2n} \ \frac{\Gamma(2n)}{\Gamma(n)^2}.
\end{eqnarray*}
So, thanks to Stirling's formula, we get
\begin{equation}
 \sum_{n=0}^{\infty} || C^n|| \leq  1 + ||C|| + \sum_{n=2}^{+\infty}  \left( \frac{C_{A, \delta}}{2(n-1)\delta}\right)^{2n} \ \frac{\Gamma(2n)}{\Gamma(n)^2} \leq C_{A, \delta}.
\end{equation}
This means that the Neumann series
\begin{equation}
B^{-1} = \sum_{n=0}^{+\infty} (-1)^n  C^n
\end{equation}
is convergent in the operator norm and $||B^{-1}|| \leq C_{A, \delta} $.
\end{proof}

\vspace{0.3cm}
Now, let us assume again that the warping functions $c, \tilde{c} \in \mathcal{C}(A)$ and that for all $k \geq 0$,
\begin{equation}\label{diffsigma}
|\sigma_k - \tilde{\sigma}_k | \leq \e.
\end{equation}
By making $k \rightarrow + \infty$ in (\ref{diffsigma}), and thanks to  Theorem \ref{AsympSteklov}, we deduce that $f(0)= \tilde{f}(0)$, and also
\begin{equation}\label{diffderivee}
\left| \frac{(d-2) f'(0)}{f^3(0)} - \frac{(d-2) \tilde{f}'(0)}{\tilde{f}^3(0)} \right| \leq \epsilon.
\end{equation}
Thus, recalling that for  sufficiently large $k$, we have
\begin{equation}
\sigma_k = \frac{(d-2) f'(0)}{f^3(0)} - \frac{M(-\K_k^2)}{f^2(0)},
\end{equation}
we obtain:
\begin{equation}\label{sigM}
| \tilde{M}(-\K^2)- M(-\K^2)| \leq  2 f^2(0) \  \epsilon \leq 2 A^2\  \epsilon,
\end{equation}
since $c \in \mathcal{C}(A)$. So, plugging (\ref{sigM}) into (\ref{intM}) and using Corollary \ref{estSinfini}, we easily get the following result:

\begin{lemma}\label{momentsexp}
Let $c, \tilde{c} \in \mathcal{C}(A)$. Assume that for all $k \geq0$, $|\sigma_k - \tilde{\sigma}_k | \leq \e$.
Then, there exists a positive constant $C_A$ which does not depend on $\e$ such that
\begin{equation} \label{firstest}
\left| \int_0^{+\infty}  e^{-2\K_k x} B[q_{\tilde{f}} -q_f] (x) \ dx \right| \leq C_A\ \e.
\end{equation}
\end{lemma}

\vspace{0.1cm}

Note that this kind of estimates fits into the so-called {\it{moment theory}}, (see for instance \cite{AGLT} for a nice exposition and references therein), and this is the object of the next sections.

\subsection{A M\"untz-Jackson's theorem.}

Let $\Lambda_{\infty} = \{ 0 \leq \lambda_0 < \lambda_1 < ...< \lambda_n < ... \}$, $\lambda_n \rightarrow + \infty$ be a sequence of positive real numbers. The classical M\"untz-Sz\' asz's Theorem characterizes the sequences $(\lambda_k)$ for which  all functions in  $C^0([0,1])$, or in $L^2 ([0,1])$, can be approximated by "M\"untz polynomials" of the form:
\begin{equation}
P(x) = \sum_{k=0}^n a_k \ x^{\lambda_k},
\end{equation}
with real coefficients $a_k$. More precisely, one has:

\begin{thm}\label{M\"untz}
Let $\Lambda_{\infty}$ be a sequence of positive real numbers as above. Then,  span $\{ x^{\lambda_0}, x^{\lambda_1}, ... \}$  is dense in $L^2([0,1])$ if and only if
\begin{equation}\label{condM\"untz}
\sum_{k=1}^{+\infty} \frac {1}{\lambda_k} = \infty.
\end{equation}
Moreover, if $\lambda_0=0$, the denseness of the M\"untz polynomials  in  $C^0([0,1])$ in the sup norm is also characterized by (\ref{condM\"untz}).
\end{thm}

\noindent
Now, for $ n\geq 1$, let us consider the finite sequence
\begin{equation}
 \Lambda:= \Lambda_n : 0 \leq \lambda_0<\lambda_1 < ... <\lambda_n.
\end{equation}
We define the subspace of the  "M\"untz polynomials of degree $\lambda_n$"  as:
\begin{equation}
 \mathcal{M}(\Lambda) = \{ P : \ P(x) = \sum_{k=0}^n a_k \ x^{\lambda_k} \}.
\end{equation}
The error of approximation from $\mathcal{M}(\Lambda)$ of a function $ f$ in $C^0([0,1])$ or in $ L^2([0,1])$ is given by:
\begin{equation}
 E(f, \Lambda)_p := \inf_{P \in \mathcal{M}(\Lambda)} \ || f-P||_p = ||f-P_0||_p,
\end{equation}
for some $P_0 \in \mathcal{M}(\Lambda)$ depending on whether $p=2$ or $p=\infty$. Clearly, one has  $E(f, \Lambda)_2 \leq E(f, \Lambda)_{\infty}$ and
$E(f, \Lambda)_2 = ||f-\pi_n(f)||_2$ where $\pi_n(f)$ is the orthogonal projection of $f$ on the subspace $  \mathcal{M}(\Lambda)$. An estimation from above of $E(f, \Lambda)_p$ in terms of
the smoothness of $f$ is called a  {\it{M\"untz-Jackson's theorem}}.

\vspace{0.2cm}\noindent
To estimate this error of approximation for the uniform norm,  let us consider the so-called Blaschke product $B(z)$, $z \in \C$,
\begin{equation}
 B(z):=B(z, \Lambda) = \prod_{k=0}^n \frac{z-\lambda_k }{z+\lambda_k}.
\end{equation}
The  {\it{index of approximation}} of $\Lambda$ in $C^0([0,1])$ is defined as :
\begin{equation}
 \e_{\infty} (\Lambda) = \max_{y \geq 0} \ \left| \frac{B (1+iy)}{1+iy} \right|.
\end{equation}
Its relevance is justified by the following result (\cite{LGM}, Theorem 2.6, Chapter 11) when $\lambda_0 =0$:

\begin{prop}\label{errorM}
Let $\Lambda : 0= \lambda_0 < \lambda_1 < ... < \lambda_n$ be a finite sequence. Then, for each $f \in C^1 ([0,1])$,
\begin{equation}
 E(f, \Lambda)_{\infty} \leq 20 \ \e_{\infty}(\Lambda) \ ||f'||_{\infty}.
\end{equation}
\end{prop}

\vspace{0.2cm}
\noindent
As a by-product, we can easily prove:

\begin{coro}\label{M\"untzC1}
Let $\Lambda^* : 0 < \lambda_1 < ... < \lambda_n$ be a finite sequence. Then, for each $f \in C^1 ([0,1])$ with $f(0)=0$, one has:
\begin{equation}
 E(f, \Lambda^*)_{\infty} \leq 40 \  \e_{\infty}(\Lambda^*) \ ||f'||_{\infty}.
\end{equation}
\end{coro}

\begin{proof}
Consider the finite sequence $\Lambda : 0= \lambda_0 < \lambda_1 < ... < \lambda_n$. Thanks to Proposition \ref{errorM}, there exists $P \in \mathcal{M}(\Lambda)$ such that
\begin{equation}
  E(f, \Lambda)_{\infty} = ||f-P||_{\infty} \leq 20 \ \e_{\infty} (\Lambda) \ ||f'||_{\infty}.
\end{equation}
Since $f(0)=0$, we deduce that  $|P(0)| \leq 20 \ \e_{\infty}(\Lambda) \ ||f'||_{\infty}$. Thus, setting $Q= P-P(0) \in \mathcal{M}(\Lambda^*)$, we obtain:
\begin{equation}
||f-Q||_{\infty} \leq 40 \ \e_{\infty}(\Lambda) \ ||f'||_{\infty},
\end{equation}
which concludes the proof since $\e_{\infty}(\Lambda)=\e_{\infty}(\Lambda^*)$.
\end{proof}

\noindent
In the same way, we can recursively approximate $C^r$-differentiable functions, ($r \geq 1$). To this end, let  $\Lambda^* :\  r-1 < \lambda_1 < ... < \lambda_n$ be a finite sequence. For $0 \leq k \leq r-1$, we set:
\begin{equation}
 \Lambda^{(k)} : \ \lambda_1^{(k)}= \lambda_1-k,\  ...\ ,\  \lambda_n^{(k)} = \lambda_n -k, \
\end{equation}
and the indices of approximation $\e_{\infty}^{(k)} = \e_{\infty} (\Lambda^{(k)}), \ k=0, 2, ... , r-1$.

\vspace{0.2cm}\noindent
We have the following result:

\begin{coro}\label{M\"untzCr}
For each $f \in C^r ([0,1])$ such that $f^{(k)}(0)=0$ for all $k =0, ..., r-1$, one has:
\begin{equation}
 E(f, \Lambda^*)_{\infty} \leq 40^r  \  \prod_{k=0}^{r-1} \e_{\infty}^{(k)} \ \ ||f^{(r)}||_{\infty}.
\end{equation}
\end{coro}

\begin{proof}
For $r=1$, this estimate is nothing but Corollary \ref{M\"untzC1}. Now, assume that $r=2$ and consider $f \in C^2 ([0,1])$ with $f(0)=f'(0)=0$. Using Corollary \ref{M\"untzC1}
for the function $f'$ and the finite sequence $\Lambda^{(1)}$, we see there exists ${\displaystyle{P(x) = \sum_{k=1}^n a_k \  x^{\lambda_k -1} \in \mathcal{M}(\Lambda^{(1)})}}$ such that
\begin{equation}\label{firstineq}
E(f', \Lambda^{(1)})_{\infty} = ||f'-P||_{\infty} \leq 40 \  \e_{\infty}^{(1)} \ ||f''||_{\infty}.
\end{equation}
We set ${\displaystyle{F(x)= f(x)- \int_0^x P(t) \ dt = f(x) - \sum_{k=1}^n \frac{a_k}{\lambda_k} \ x^{\lambda_k}}}$. Since $F(0)=0$, using again Corollary \ref{M\"untzC1} for the
finite sequence $\Lambda^*$, we see there exists $Q \in \mathcal{M}(\Lambda^*)$ such that
\begin{equation}
E(F, \Lambda^*)_{\infty} = ||F-Q||_{\infty} \leq 40 \  \e_{\infty} \ ||F'||_{\infty},
\end{equation}
or equivalently
\begin{equation}
||f - (\int_0^x P(t) \ dt +Q) || \leq 40 \  \e_{\infty} \ ||f'-P||_{\infty}.
\end{equation}
Observing that ${\displaystyle{\int_0^x P(t) \ dt +Q(x) = \sum_{k=1}^n \frac{a_k}{\lambda_k} x^{\lambda_k} +Q(x) \in \mathcal{M}(\Lambda^*)}}$, and using (\ref{firstineq}), we obtain:
\begin{equation}
 E(f, \Lambda^*)_{\infty} \leq ||f - (\int_0^x P(t) \ dt +Q) || \leq 40^2 \ \e_{\infty} \  \e_{\infty}^{(1)} \ ||f''||_{\infty},
\end{equation}
which proves Corollary \ref{M\"untzCr} in the case $r=2$. For $r \geq 3$, the proof is identical.
\end{proof}

\vspace{0.5cm}
For special finite sequences $\Lambda$, the index of approximation $\e_{\infty}(\Lambda)$ can be replaced by a much simpler expression.
For instance, we have the following result, (\cite{LGM}, Theorem 4.1, Chapter 11):

\vspace{0.1cm}

\begin{thm}\label{tailleepsilon}
Let $\Lambda :  0= \lambda_0 \ <  \lambda_1< \lambda_2< ...<\lambda_n$ be a finite sequence. Assume that  $\lambda_{k+1}-\lambda_k \geq 2$ for  $ k \geq 0$. Then,
\begin{equation}
\e_{\infty} (\Lambda) = |B(1, \Lambda) | = \prod_{k=1}^n \frac{\lambda_k -1} {\lambda_k +1}
\end{equation}
In particular, if  $\lambda_k = 2k+b$ for $k=1, ..., n$  where $b>0$, one has
\begin{equation}
\e_{\infty} (\Lambda) = \frac{b+1}{2n+b+1}.
\end{equation}
\end{thm}

\subsection{A Hausdorff moment problem with non-integral powers.}

\vspace{0.5cm}\noindent
As an application of the previous section, we shall give an approximation of the $L^2$-norm of a function $f$ from the approximately knowledge of a finite number of these moments:
\begin{equation}\label{moments}
 m_k = \int_0^1 \ t^{\lambda_k} \ f(t) \ dt \ ,\ k \in X \subset \N.
\end{equation}
Thanks to Theorem \ref{M\"untz}, if $\Lambda_{\infty} = \{ 0 \leq \lambda_0 < \lambda_1 < ...< \lambda_n < ... \}$, $\lambda_n \rightarrow + \infty$ is a sequence of positive real numbers such that
\begin{equation}
\sum_{k=1}^{+\infty} \frac {1}{\lambda_k} = \infty,
\end{equation}
the system $\{ x^{\lambda_0}, \ x^{\lambda_1}, ... \}$ is complete in $L^2([0,1])$. Thus, the knowledge of the complete sequence ${\displaystyle{(m_k)_{k \geq 0}}}$, uniquely determines the function $f$. But, in practice, one has available only a finite set $m_0, ..., m_n$ of moments, and furthermore these moments are usually corrupted with noise. It is well-known that this problem is severely ill-posed, (see for instance \cite{AGLT} and references therein, for further details).

\vspace{0.2cm} \noindent
Let us briefly explain the approach given in \cite{AGLT}. Using the Gram-Schmidt process, we define the polynomials $(L_m(x))$ as $L_0(x)=1$, and for $m \geq 1$,
\begin{equation}\label{Lm}
L_m(x) = \sum_{j=0}^m C_{mj} x^{\lambda_j},
\end{equation}
where we have set
\begin{equation}\label{cm}
C_{mj}= \sqrt{2\lambda_m +1} \  \frac{ \prod_{r=0}^{m-1} (\lambda_j + \lambda_r +1)}{ \prod_{r=0, r \not= j}^{m} (\lambda_j - \lambda_r)}.
\end{equation}
The family $(L_m (x))$ defines an orthonormal Hilbert basis of $L^2([0,1])$. For instance, if $\lambda_k = k$, the polynomials $(L_m(x))$ are the Legendre polynomials, and in this case, the latter coefficients $C_{mj}$ are given by:
\begin{equation}\label{coefLegendre}
C_{mj}^0 : =  \sqrt{2m+1} \ (-1)^{m-j} \ \frac{(m+j)!}{(m-j)! \ j!^2}
\end{equation}
Note that, the generalized binomial theorem gives easily the upper bound:
\begin{equation}\label{estCmj}
| C_{mj}^0 | \leq \sqrt{2m+1} \ 3^{m+j}.
\end{equation}

\vspace{0.2cm} \noindent
Now, let us consider the finite sequence
\begin{equation}
 \Lambda:= \Lambda_n : 0 \leq \lambda_0 <\lambda_1 < ... <\lambda_n.
\end{equation}
Assume that the $(n+1)$ first moments of a function $f \in L^2([0,1])$ are equal to zero up to noise, i.e there exists $\e>0$ such that
\begin{equation}\label{bruit}
| \int_0^1 f(t) \ t^{\lambda_k} \ dt | \leq \e  \ ,\ \forall k=0, ..., n.
\end{equation}
We denote $\pi_n (f)$ the orthogonal projection on on the subspace $\mathcal{M}(\Lambda)$:
\begin{equation}
\pi_n(f) = \sum_{k=0}^n <f, L_k > L_k.
\end{equation}
Thus, we deduce that:
\begin{eqnarray}\label{inegproj}
|| \pi_n (f) ||_2^2 & = & \sum_{k=0}^n | <f, \sum_{p=0}^k C_{kp} \  x^{\lambda_p} > |^2 \nonumber \\
                    &\leq & \e^2 \ \sum_{k=0}^n \left( \sum_{p=0}^k |C_{kp}| \right)^2,
\end{eqnarray}
thanks to our hypothesis on the moments (\ref{bruit}). So, we get immediately:
\begin{eqnarray*}
||f||_2^2       &=& || \pi_n (f) ||_2^2  + ||f- \pi_n (f) ||_2^2 \\
                &=& || \pi_n (f) ||_2^2 + E(f, \Lambda)_2^2\\
                & \leq& \e^2 \ \sum_{k=0}^n \left( \sum_{p=0}^k |C_{kp}| \right)^2  + E(f, \Lambda)_{\infty}^2 \\
                & \leq& \e^2 \ \sum_{k=0}^n \left( \sum_{p=0}^k |C_{kp}| \right)^2  + E(f, \Lambda^*)_{\infty}^2,
\end{eqnarray*}
where $\Lambda^* : 0 < \lambda_1 < ... < \lambda_n$. In particular, if $f \in C^r([0,1]$ with $f^{(k)}(0)=0$ for $k =0, ..., r-1$, and if $r-1< \lambda_1$, we get
using Corollary \ref{M\"untzCr}:
\begin{equation} \label{balance}
||f||_2^2   \leq  \e^2 \ \sum_{k=0}^n \left( \sum_{p=0}^k |C_{kp}| \right)^2 + \left( 40^r  \  \prod_{k=0}^{r-1} \e_{\infty}^{(k)} \ \ ||f^{(r)}||_{\infty} \right)^2.
\end{equation}

At this stage, it is important to make the following remark: on the one hand, the double sum appearing  in the (RHS) of (\ref{balance}) can be very large with respect to $n$.
Indeed, in the case $\lambda_k=k$, the coefficients $C_{k0}^0$ are equal to  $\sqrt{2k+1}$ thanks to  (\ref{coefLegendre}). On the other hand, if $\lambda_k = 2k$,
Theorem \ref{tailleepsilon} suggests that the second term in the (RHS) of (\ref{balance}) is equal to $O(n^{-r})$, $n \rightarrow +\infty$. Thus, if we want to control
reasonably (with respect to $\e$) the $L^2$-norm of the function $f$, we have to choose a suitable $n=n(\e)$ in the equation (\ref{balance}). Of course, this choice will depend heavily of the behaviour of the coefficients $C_{kp}$. This will be done in the next two sections where we separate the simpler case
where the metric in regular, and the case where the metric is singular at $r=0$.

 \subsection{The regular case.}

 \subsubsection{Stability estimates for the potentials.}

In this case, we recall that necessarily, the metric $g_S = d\Omega^2$ and  $\kappa_k = k + \frac{d-2}{2}$, (we order the Steklov spectrum without counting multiplicity).
So, making the change of variables $t=e^{-x}$ in (\ref{firstest}), we obtain:
\begin{equation} \label{moment}
| \int_0^1  t^{2k+d-3} \ B[q_{\tilde{f}} -q_f] (-\log t)  \ dt | \leq C_A\ \e, \ \forall k \geq 0.
\end{equation}
Thus,  introducing $\delta \in ]0,1[$, we get:
\begin{equation} \label{moment1}
| \int_0^1  t^{\lambda_k} \ h(t) \ dt | \leq C_A\ \e, \ \forall k \geq 0,
\end{equation}
where we have set $h(t) = t^{- \frac{\delta+1}{2}} B[q_{\tilde{f}} -q_f] (-\log t)$ and $\lambda_k = 2k+d-3 + \frac{\delta+1}{2}$. Note that
\begin{equation}\label{L2h}
||h||_{L^2(0,1)} = || B[q_{\tilde{f}} - q_f] ||_{\mathcal{H}_{\delta}}.
\end{equation}

\vspace{0.2cm}\noindent
Using Lemma \ref{derivK1}, we see that $h(t) \in C^{m-2}((0,1])$ and we have the following result:

\begin{lemma}
The function  $t \to h(t)$ extends in a $C^{p-1}$-differentiable function on $[0,1]$ with $h^{(k)}(0)=0$ for $k=0, ..., p-1$.  Moreover, there exists a constant $C_A>0$ such that, for all $k \leq p-1$,
\begin{equation}\label{estderiveeh}
|| h^{(k)} ||_{\infty} \leq C_A.
\end{equation}
\end{lemma}

\begin{proof}
We only sketch the proof since the arguments are straightforward. Recalling that $B = Id + C$, we write :
\begin{eqnarray*}
h(t) &= & t^{- \frac{\delta+1}{2}} [q_{\tilde{f}} -q_f] (-\log t) + t^{- \frac{\delta+1}{2}} C[q_{\tilde{f}} -q_f] (-\log t) \\
     &:=& h_1(t) + h_2(t).
\end{eqnarray*}
First, let us estimate the derivatives of $h_1(t)$. Setting $Q(x):= [q_{\tilde{f}} -q_f](x)$, we deduce immediately from (\ref{estqA}) that:
\begin{equation}
|Q^{(k)} (x) | \leq C_A \ e^{-px}, \forall x \geq 0, \ \forall k = 0, ...,m-2.
\end{equation}
Then, we get  easily:
\begin{equation}\label{esth1}
 | h_1^{(k)} (t) | \leq C_A \ t^{p-k- \frac{\delta+1}{2}} \ ,\ \forall k =0, ..., m-2, \ \ \forall t \in ]0,1].
\end{equation}
In the same way, using again Lemma \ref{derivK1}, a tedious calculation shows that for all $\alpha <p$, there exists a constant $C_{A,\alpha}>0$ such that
\begin{equation}\label{esth2}
 | h_2^{(k)} (t) | \leq C_{A, \alpha} \ t^{\alpha-k- \frac{\delta+1}{2}} \ ,\ \forall k =0, ..., m-2.
\end{equation}
Then, it follows from (\ref{esth1}) and (\ref{esth2}) that the estimate (\ref{estderiveeh}) is satisfied. Moreover, for all $k=0, ..., p-1$, we see that $h^{(k)}(t) \rightarrow 0$ as $t \rightarrow 0$. This concludes the proof.
\end{proof}

\vspace{0.5cm}
\noindent
Now, let us estimate the M\"untz coefficients $C_{mj}$ associated with  $\lambda_k = 2k+d-3 + \frac{\delta+1}{2}$. To simplify the notation we set $b = d-3 + \frac{\delta+1}{2}$.
Using (\ref{cm}), we get:
\begin{equation}
 C_{mj} = \sqrt{4m+2b+1} \ \frac{ \prod_{r=0}^{m-1} (2j+2r+4b+1)}{\prod_{r=0, r \not=j}^m (2j-2r)}.
\end{equation}
Setting $M = \max (2, 4b+1)$, we get easily:
\begin{equation}
 |C_{mj} | \leq \sqrt{4m + \frac{M+1}{2}} \ \left(\frac{M}{2}\right)^m \ \left| \frac{ \prod_{r=0}^{m-1} (j+r+1)}{\prod_{r=0, r \not=j}^m (j-r)} \right|,
\end{equation}
or equivalently
\begin{equation}
 |C_{mj} | \leq \sqrt{\frac{4m + \frac{M+1}{2}}{2m+1}} \ \left(\frac{M}{2}\right)^m \ | C_{mj}^0 |.
\end{equation}
Thus, there exists a universal constant $B>0$ such that
\begin{equation}\label{estCmjreg}
 |C_{mj} | \leq B  \ \left(\frac{M}{2}\right)^m \ |C_{mj}^0|.
\end{equation}

\noindent
Thanks to (\ref{estCmjreg}),  we can estimate the $L^2$-norm of the ortogonal projection $\pi_n (h)$ on the subspace $\mathcal{M}(\Lambda)$ where $\Lambda = \Lambda_n: 0< \lambda_0 < \lambda_1 < ... <\lambda_n$, where
$\lambda_k = 2k+d-3 + \frac{\delta+1}{2}$:
\begin{eqnarray*}
||\pi_n (h)||_2^2   &\leq & \e^2 \ \sum_{k=0}^n \left( \sum_{p=0}^k |C_{kp}| \right)^2  \\
                  &\leq & \e^2 \ \sum_{k=0}^n \left( \sum_{p=0}^k  \left(\frac{M}{2}\right)^k \ |C_{kp}^0| \right)^2 \\
                   &\leq & B^2 \e^2 \ \sum_{k=0}^n \left(\frac{M}{2}\right)^k \left( \sum_{p=0}^k   \ |C_{kp}^0| \right)^2.
\end{eqnarray*}
Using (\ref{estCmj}), we see that
\begin{equation}
 \sum_{p=0}^k   \ |C_{kp}^0| \leq \sqrt{2k+1} \sum_{p=0}^k 3^{k+p} \leq  \frac{3}{2} \ \sqrt{2k+1} \ 3^{2k}.
\end{equation}
We deduce that

\begin{eqnarray*}
 \sum_{k=0}^n \left(\frac{M}{2}\right)^k \left( \sum_{p=0}^k   \ |C_{kp}^0| \right)^2 &\leq & \ \frac{9}{4} \sum_{k=0}^n \left(\frac{M}{2}\right)^{2k} (2k+1) 3^{4k} \\
                 &\leq &  \ \frac{9}{4} \ (2n+1) \   \sum_{k=0}^n \left(\frac{9M}{2}\right)^{2k}  \\
                 &\leq &  \ \frac{9}{4}  \  (2n+1)  \frac {  \left(\frac{9M}{2}\right)^{2n+2}   } {  \left(\frac{9M}{2}\right)^{2}-1} := g(n)^2,
\end{eqnarray*}
where $ g :[0, + \infty[$ is the strictly increasing function defined for $t \in [0,+\infty[$ as
\begin{equation}
 g(t) = \frac{3}{2}   \frac {  1   } {\sqrt{  \left(\frac{9M}{2}\right)^{2}-1}} \ \sqrt{2t+1} \ \left(\frac{9M}{2}\right)^{t+1}.
\end{equation}
At this stage, we have obtained
\begin{equation}\label{estproj1}
  ||\Pi_{n} h||_2^2 \leq B^2 \e^2 \ g(n)^2.
\end{equation}

\noindent
Now, let us choose a suitable integer $n$ to control properly the norm of the projection $||\pi_{n} h||_2^2$. We set ${\displaystyle{n(\epsilon) := E \ (g^{-1} ( \frac{1}{\sqrt{\e}}))}}$. Clearly, since $g$ is an increasing function, one has $g(n(\e)) \leq \frac{1}{\sqrt{\e}}$, and thanks to (\ref{estproj1}) we get immediately:
\begin{equation}\label{stability1}
 ||\pi_{n(\e)} h||_2^2 \leq B^2 \e.
\end{equation}

\vspace{0.5cm}
It remains to estimate $||h-\pi_{n(\e)} h||_2= E(h, \Lambda_{n(\e)})_2$. First, let us introduce for $\e$ small enough,
\begin{equation}
\tilde{\Lambda}_{n(\e)} : \lambda_{k_0} < \lambda_{k_0+1}< ... < \lambda_{n(\e)},
\end{equation}
where $\lambda_{k_0}>p-2$. Obviously, $E(h, \Lambda_{n(\e)})_2 \leq E(h, \tilde{\Lambda}_{n(\e)})_2 \leq E(h, \tilde{\Lambda}_{n(\e)}))_{\infty}$. So using Corollary
\ref{M\"untzCr} with $r=p-1$, we obtain:
\begin{equation}
E(h,\tilde{\Lambda}_{n(\e)})_{\infty} \leq 40^{p-1}  \  \prod_{k=0}^{p-2} \e_{\infty}^{(k)} \ \ ||h^{(p-1)}||_{\infty},
\end{equation}
where $\e_{\infty}^{(k)}= \e_{\infty}(\tilde{\Lambda}_{n(\e)}^{(k)})$. Since $\lambda_{j+1}^{(k)} - \lambda_{j}^{(k)}=2$, we can use Theorem \ref{tailleepsilon}, and thanks to (\ref{estderiveeh}), we get easily:
\begin{equation}
E(h,\tilde{\Lambda}_{n(\e)})_{\infty} \leq C_A  \  \left(\frac{1}{n(\e)}\right)^{p-1}.
\end{equation}
Now, a straightforward calculation shows that
\begin{equation}
 n(\e) \sim C \log ( \frac{1}{\e} ), \ \e \rightarrow 0,
\end{equation}
for a suitable constant $C>0$. Thus, we have proved
\begin{equation}\label{stability2}
E(h,\tilde{\Lambda}_{n(\e)})_{\infty} \leq C_A  \  \left(\frac{1}{ \log ( \frac{1}{\e} )}\right)^{p-1}.
\end{equation}
As a conclusion, thanks to (\ref{L2h}), (\ref{stability1})  and (\ref{stability2}), we have obtained  for $\e$ small enough:
\begin{equation}\label{stableh}
|| B[q_{\tilde{f}} - q_f ]||_{\mathcal{H}_{\delta}} \leq C_A \left(\frac{1}{ \log ( \frac{1}{\e} )}\right)^{p-1}.
\end{equation}
By Proposition \ref{isomorphism}, $B : \mathcal{H}_{\delta} \rightarrow \mathcal{H}_{\delta}$ is an isomorphism and $||B^{-1} || \leq C_A$, then we get a stability estimate between the two potentials $q_{\tilde{f}}$ and $q_f$ for the topology of ${\mathcal{H}_{\delta}}$:
\begin{equation}\label{stableq}
|| q_{\tilde{f}} - q_f ||_{\mathcal{H}_{\delta}} \leq C_A \left(\frac{1}{ \log ( \frac{1}{\e} )}\right)^{p-1}.
\end{equation}

\subsubsection{Stability estimates for the warping functions.}

\vspace{0.5cm}\noindent
First, let us  prove a stability estimate for  the warping functions $\tilde{f}$ and $f$ in the variable $x \in (0, +\infty)$. In order to simplify the notation,
we set $F = f^{d-2}, \ \tilde{F} = \tilde{f}^{d-2}$.  We we write:
\begin{eqnarray*}
(\tilde{F}'F - F' \tilde{F})' (y) &=&  F \tilde{F} \ (q_{\tilde{f}} -q_f) (y) \\
                                  &\leq& C_A \ e^{-(d-2)y} \ |(q_{\tilde{f}} -q_f) (y)| ,
\end{eqnarray*}
since for instance, $f(y)= c(e^{-y}) e^{- \frac{y}{2}}$ and $c \in \mathcal{C}(A)$. Integrating this last inequality on the interval $(x, \infty)$, and setting $G = \tilde{F}'F - F' \tilde{F}$,
we get:
\begin{eqnarray*}\label{estG}
 |G(x)| &\leq& C_A \ e^{-(d-2)x} \  \int_x^{+\infty}  |(q_{\tilde{f}} -q_f) (y)| \ dy \\
        &\leq& C_A \ e^{-(d-2)x} \  \int_x^{+\infty} e^{- \frac{\delta}{2} y} \ e^{\frac{\delta}{2} y} |(q_{\tilde{f}} -q_f) (y)| \ dy \\
        &\leq& C_A \ e^{-(d-2+\frac{\delta}{2})x} \  || q_{\tilde{f}} - q_f ||_{\mathcal{H}_{\delta}} \\
        &\leq& C_A \ e^{-(d-2+\frac{\delta}{2})x} \  \left(\frac{1}{ \log ( \frac{1}{\e} )}\right)^{p-1},
\end{eqnarray*}
where we have used (\ref{stableq}) and the Cauchy-Schwarz inequality.
Writing $\left( \frac{ \tilde{F}}{F} \right)' = \frac{G}{F^2}$, and since $c \in \mathcal{C}(A)$, we deduce from (\ref{estG}) that
\begin{eqnarray*}
 \left( \frac{ \tilde{F}}{F} \right)' (y) &\leq& C_A \ e^{(d-2)y} \  e^{-(d-2+\frac{\delta}{2})y} \left(\frac{1}{ \log ( \frac{1}{\e} )}\right)^{p-1} \\
                                          &\leq& C_A \  \  e^{-\frac{\delta}{2} y} \left(\frac{1}{ \log ( \frac{1}{\e} )}\right)^{p-1}.
\end{eqnarray*}
 We integrate again this equality on $(x, \infty)$ and we get:
\begin{equation}\label{estfraction}
 |1 - \frac{ \tilde{F}}{F} (x) | \leq C_A \ e^{-\frac{\delta}{2} x} \ \left(\frac{1}{ \log ( \frac{1}{\e} )}\right)^{p-1}
\end{equation}
Then, we deduce easily that, for all $x \geq 0$,
\begin{equation}
| \tilde{f}(x) - f(x) | \leq C_A \ e^{-\frac{\delta+1}{2} x} \ \left(\frac{1}{ \log ( \frac{1}{\e} )}\right)^{p-1}.
\end{equation}
As a consequence, in the variable $r=e^{-x} \in (0,1]$, we get:
\begin{equation}
| \tilde{c}(r) - c(r) | \leq C_A \ r^{\frac{\delta}{2} } \ \left(\frac{1}{ \log ( \frac{1}{\e} )}\right)^{p-1},
\end{equation}
and the proof of Theorem \ref{LogStabSteklov} is now complete.

\subsection{The singular case.}

\vspace{0.2cm}
We only sketch the proof since it is very similar to the previous one in the regular case. The real difference lies in the fact that  we have no explicit formula for the angular eigenvalues $\kappa_k$,
and thus
we have to use the Weyl asymptotics (\ref{WeylLaw}). First, let us consider the sub-sequence $\nu_k = \kappa_{k^{d-1}}$. Thanks to Lemma \ref{momentsexp}, we get:
\begin{equation} \label{firstest2}
\left| \int_0^{+\infty}  e^{-2\nu_k x} B[q_{\tilde{f}} -q_f] (x) \ dx \right| \leq C_A\ \e \ ,\ \forall k \geq 0.
\end{equation}
Thus, making the change of variables $t = e^{-x}$, we obtain immediately:
\begin{equation} \label{firstest3}
\left| \int_0^1  t^{2 \nu_k -1}  B[q_{\tilde{f}} -q_f] (- \log t) \ dt \right| \leq C_A\ \e \ ,\ \forall k \geq 0.
\end{equation}
Now, let us choose $\alpha \in ] \frac{1}{2}, \frac{3}{2}[$ and let $N$ be an integer large enough which we shall specify below.  We deduce easily from (\ref{firstest3}) that:
\begin{equation} \label{firstest4}
\left| \int_0^1  t^{\lambda_k}  h(t) \ dt \right| \leq C_A\ \e \ ,\ \forall k \geq 0,
\end{equation}
where we have set $h(t) = t^{-\alpha} B[q_{\tilde{f}} -q_f] (- \log t)$ and $\lambda_k =  2\nu_{Nk} + \alpha -1$. Note  that
$||h||_{L^2 (0,1)} = || B[q_{\tilde{f}} -q_f] ||_{\mathcal{H}_{\delta}}$ with $\delta= (2\alpha-1) \in (0,1[$.

\vspace{0.5cm}
First, let us verify that we are in the separate case in the framework of the M\"untz-Jackson approximation,  i.e $\lambda_{k+1}-\lambda_k \geq 2$.
It follows from the Weyl's law (\ref{WeylLaw}) that there exists $C \geq 1$  such that:
\begin{equation}
| \nu_{k} - c_{d-1} k | \ \leq \  C \ ,\ \forall k \geq 0.
\end{equation}
Now, let us fix an integer $N$ large enough such that $B:=c_{d-1}N >3C$. One deduces that
\begin{eqnarray*}
\lambda_{k+1} - \lambda_k & =    & 2 [ \nu_{(k+1)N} - \nu_{kN} ] \\
                          & \geq & 2 [(B(k+1) -C) -(Bk+C) ] \\
                          & \geq & 2B-4C \geq 2C \geq 2.
\end{eqnarray*}

\vspace{0.5cm}
Second, let us  estimate the index of approximation $\epsilon_{\infty}(\Lambda)$. We recall that
\begin{equation}
\epsilon_{\infty}(\Lambda) = \prod_{k=1}^n \frac{\lambda_k -1}{\lambda_k +1} =  \prod_{k=1}^n \left( 1 - \frac{2}{\lambda_k +1} \right).
\end{equation}
Since $\lambda_k >1$, we easily get:
\begin{eqnarray*}
\log \epsilon_{\infty}(\Lambda) &=& \sum_{k=1}^n \log \ ( 1 - \frac{2}{\lambda_k +1} )\\
                              & \leq & -2 \ \sum_{k=1}^n \frac{1}{\lambda_k +1} \\
                              & \leq &  -2  \ \sum_{k=1}^n \frac{1}{2(Bk+C) + \alpha}.
\end{eqnarray*}
Then, we deduce the following estimate :
\begin{equation}
\epsilon_{\infty}(\Lambda) = O \left( \frac{1}{n^{\frac{1}{B}}} \right).
\end{equation}

\vspace{0.5cm}
Finally, we have to estimate the M\"untz coefficients $C_{mj}$ associated with our  sequence $\lambda_k$ as in the previous section in the regular case. We recall that
\begin{equation}\label{cm1}
C_{mj}= \sqrt{2\lambda_m +1} \  \frac{ \prod_{r=0}^{m-1} (\lambda_j + \lambda_r +1)}{ \prod_{r=0, r \not= j}^{m} (\lambda_j - \lambda_r)}.
\end{equation}
Clearly, for $j >r$ (for instance), one has $\lambda_j - \lambda_r = 2 (\nu_{kN} - \nu_ {jN}) \geq 2 [ B(j-r) -2C]$. Thus, choosing $b \in (0, B-2C]$, we get immediately
$\lambda_j - \lambda_r \geq b (j-r)$. We deduce that:
\begin{equation}
 | \prod_{r=0, r \not= j}^{m} (\lambda_j - \lambda_r) | \ \geq \ (2b)^m \prod_{r=0, r \not= j}^{m} |j-r|
\end{equation}
In the same way, one has:
\begin{eqnarray*}
 |\lambda_j + \lambda_r | &=& | \nu_{kj} + \nu_{kr} + 2 \alpha -1 | \\
                          &\geq& |B(j+r) +2C + 2\alpha-1 | \\
                          &\geq& M \ |j+r+1 |,
\end{eqnarray*}
where $M = max \{B, 2C+2\alpha -1\}$. It follows there exists $D>0$ such that
\begin{equation}
 | C_{mj} | \leq D \ \left(\frac{M}{2b}\right)^m \ |C_{mj}^0 |.
\end{equation}

\vspace{0.5cm}
Now, following exactly the same approach as in the regular case, and taking $\theta = \frac{1}{B}\in (0,1[$, we get Theorem \ref{LogStabSteklovsing}. The details are left to the reader.



\end{document}